\documentclass[12pt]{article}
\usepackage{setspace}
\usepackage{graphicx}
\usepackage{amssymb}
\usepackage{amsmath}
\usepackage{amsthm}
\usepackage[utf8]{inputenc}
\usepackage[english]{babel}
\providecommand{\keywords}[1]
{
  \small	
  \textbf{\textit{Keywords---}} #1
}

\newtheorem{theorem}{Theorem}
\newtheorem{lemma}{Lemma}
\newtheorem{corollary}{Corollary}

\newcommand{\B}{\mathcal{B}}

\title{JORDAN PRODUCT AND ANALYTIC CORE PRESERVERS}
\author{S. Elouazzani \\ \textit{Departement of Mathematics, Labo LIABM,} \\ \textit{Faculty of Sciences, 60000 Oujda, Morocco} \\ \textit{ elouazzani.soufiane@ump.ac.ma}
\and 
M. Elhodaibi \\ \textit{Departement of Mathematics, Labo LIABM,}\\ \textit{Faculty of Sciences, 60000 Oujda, Morocco} \\ m.elhodaibi@ump.ac.ma}
\date{ }

\begin{document}
\maketitle

\textbf{Abstract} Let $\mathcal{B} (X)$ be the algebra of all bounded linear operators on an infinite-dimensional complex Banach space $X$. For an operator $ T \in \mathcal{B} (X)$, $K(T)$ denotes as usual the analytic core of $T$. We determine the form of surjective maps $ \phi $ on $ \mathcal{B} (X)$ satisfying $$ K(\phi(T) \phi(S) + \phi (S) \phi (T)) = K(TS + ST) $$ for all $T, S \in \mathcal{B} (X)$.

\section{Introduction}
Throughout this paper, $\mathcal{B} (X)$ denotes the algebra of all bounded linear operators on an infinite-dimensional complex Banach space $X$. Recall that the local resolvent set $\rho_{T}(x)$ of an operator $ T\in \mathcal{B} (X)$ at a point $x \in X$ is the set of all $\lambda$ in $\mathbb{C}$ for which there exist an open neighborhood $U_\lambda$ of $\lambda$ in $\mathbb{C}$ and an X-valued analytic function $ f : U_\lambda \rightarrow X$ such that $(T-\mu I)f(\mu) = x$ for all $\mu \in U_\lambda$. As it is known, $\sigma_T (x)$ denotes the local spectrum of $T$ at $x$, it is the complement in $\mathbb{C}$ of $\rho_{T}(x)$ and is a compact subset of the spectrum $\sigma (T)$ of $T$, possibly empty. An operator $ T \in \mathcal{B} (X) $ is said to have $the$ $single-valued$ $extension$ $property$ (SVEP) if for every open subset $U$ of $\mathbb{C}$, the equation $(T-\lambda I)f(\lambda) = 0$ for all $\lambda \in U$, has no nontrivial analytic solution f on $U$. If T has the SVEP, then $\sigma_{T}(x)\neq \emptyset $ for all nonzero vector $ x \in X$.  
Let $r$ be a positive scalar, we denote by $D (0,r)$ and $ \overline{D}(0,r)$, respectively, the open and the closed disc centered at the origin with radius $r$. For every closed subset $F$ of $\mathbb{C}$, the glocal spectral of an operator $T \in \mathcal{B} (X)$ is defined by $\mathcal{X}_{T}(F):  = \lbrace x \in X :\mbox{there exists an analytic function } f : \mathbb{C}\setminus F \longrightarrow X \mbox{ such that } (T - \lambda)f(\lambda) = x \mbox{ for all } \lambda \in \mathbb{C}\setminus F \}$. The local spectral radius of $T$ at $x \in X$ is defined by $$ r_{T}(x)= \limsup_{n\rightarrow \infty} \| T^{n} (x) \| ^{\frac{1}{n}},$$ is coincides with $ r_{T}(x) = \inf \lbrace r \geq 0: x \in \mathcal{X}_{T}(\overline{D}(0,r)) \rbrace$, and also with the maximum modulus of $\sigma_{T}(x)$ if $T$ has the SVEP, see for instance \cite{ref21}. Recall that for an operator $T \in \mathcal{B} (X)$, the analytical core $K(T)$ of $T$ is the set of all $ x\in X$ for which there exist $\delta > 0 $ and a sequence $(x_{n}) \subset X $ such that $x_{0}=x, Tx_{n+1}=x_{n}$ and $\|x_{n}\| \leq \delta^{n} \|x\|$ for all $n\geq 0$, for more information, see \cite{ref3,ref21}. Recall also that the inner local spectral radius $i_{T}(x)$ of $T$ at $x \in X$ is defined by $i_{T}(x):  = \sup \lbrace r \geq 0 : x \in \mathcal{X} _{T}(\mathbb{C}\backslash D(0,r))\rbrace ;$ see \cite{ref23}.\\
The study of local spectra preserver problems was initiated by A. Bourhim and T. Ransford \cite{ref13}, they characterized additive maps on $ \mathcal{B} (X)$ which preserve the local spectrum of all operators $ T \in \mathcal{B} (X)$ at each vector $ x\in X$. After that, the maps preserving local spectrum and local spectral radius have been studied by many authors; see \cite{ref5,ref8,ref10,ref14,ref16}. In \cite{ref20}, M.E. El Kettani and H. Benbouziane showed that a surjective additive map $ \phi $ on $ \mathcal{B} (X) $ satisfies 
$$
 i_{T}(x) = 0 \quad\mbox{ if and only if  } \;\;\; i_{\phi (T)}(x)=0
$$
for all $x\in X$ and $ T \in \mathcal{B} (X)$ if and only if there exists a nonzero scalar $c\in \mathbb{C}$ such that $\phi (T) = c T $ for all $ T \in \mathcal{B} (X)$. M. Elhodaibi and A. Jaatit, in \cite{ref19}, proved the result of T. Jari \cite{ref22} for all maps $\phi $ on $ \mathcal{B} (X)$ (not necessarily surjective) satisfying $$ i_{\phi(T)-\phi(S)}(x) = 0 \Longleftrightarrow i_{T-S}(x) = 0 $$ for every $x \in X$ and $T$, $S \in \mathcal{B} (X)$ if and only if there exists a nonzero scalar $c \in \mathbb{C}$ such that $\phi (T) = cT + \phi (0)$. For the inner local spectral radius preservers of generalized product, A. Achchi \cite{ref2}, showed that if a surjective map $\phi$ on $\mathcal{B} (X)$ satisfies $$ i_{\phi(T_{1})*...*\phi(T_{k})}(x) = 0 \Longleftrightarrow i_{T_{k}*...*T_{k}}(x) = 0 $$
for all $ x \in X $ and all $T_{1},...,T_{k} \in \mathcal{B} (X)$, then there exists a map $\gamma: \mathcal{B} (X)\longrightarrow \mathbb{C} \backslash \lbrace 0 \rbrace$ such that $ \phi (T) = \gamma (T) T $. Let us mention other papers that characterize nonlinear maps preserving certain spectral and local spectral quantities of product and Jordan product of operators or matrices, see for instance \cite{ref6,ref7,ref9,ref12}. 
In the present paper, we characterize surjective maps $ \phi$ (not necessarily linear) on $ \mathcal{B} (X)$ that satisfying
$$K (TS + ST) = K(\phi (T)\phi (S) + \phi (S)\phi (T))$$
for all $T, S \in \mathcal{B} (X)$. This gives immediately the characterization of surjective maps $ \phi : \mathcal{B} (X) \longrightarrow \mathcal{B} (X)$ satisfying   
  $$i_{TS + ST}(x) = 0 \quad\mbox{ if and only if  } \;\;\; i_{\phi (T)\phi (S) + \phi (S)\phi (T)}(x)=0 $$ 
for all $x \in X$ and $T, S \in \mathcal{B} (X)$. 

\section{Preliminaries and Notations}
For every vector $ x \in X $ and every linear functional $ f \in X^{*}$, where $X^{*}$ is the topological dual space of $X$. Denote by $ x \otimes f $ the operator of rank at most one, with $ (x \otimes f)(z) = f(z)x $ for all $ z \in X$. Recall that $ x \otimes f$ is nilpotent if and only if $f(x)= 0$, and it is idempotent if and only if $f(x)= 1$. Let $ \mathcal{F}_1 (X)$ denotes the set of all rank at most one operators on X. For a subspace A of X, the subspace spanned by $A$ is denoted by $span(A)$. Observe that, if $f(x)=0$ then $K(x \otimes f)= \lbrace 0 \rbrace$, if not then $ K(x \otimes f)$ $=span\lbrace x \rbrace$. Let $\dim(Y)$ be the dimension of a subspace $Y$ of $X$, and let for every operator $T \in \mathcal{B}(X)$, $ N(T)$ be the kernel of $T$ and $ R(T)$ be its range. \\\\
In The next lemma, we give some basic properties of $K(T)$ and $i_T (x)$, see for instance \cite{ref3,ref21,ref23}.
\begin{lemma}\label{lem1}
Let $ T \in \mathcal{B} (X) $, then the following statements hold.
\begin{enumerate}
\item[(i)] $ K(T) \subset R (T) $.
\item[(ii)] $ K (\lambda T) = K (T)$ for all nonzero scalar $ \lambda \in \mathbb{C}$.
\item[(iii)] If $M$ is a closed subspace of $X$ and $TM = M$ then $ M \subset K (T)$.
\item[(iv)] If $T$ is quasi-nilpotent then $K(T)= \lbrace 0 \rbrace $.
\item [(v)] $N(T - \lambda) \subset K(T) $ for all nonzero scalar $ \lambda \in \mathbb{C}$. 
\item [(vi)] $ K(T) = \lbrace x\in X: 0 \not\in \sigma_T (x) \rbrace$ for all $T \in \mathcal{B} (X)$.
\item [(vii)]$ i_{T}(x) = 0\quad\mbox{if and only if } \; \; 0 \in \sigma_{T}(x).$
\end{enumerate}
\end{lemma}
The next lemma and its proof are quoted in \cite{ref17}.
\begin{lemma}\label{lemme2} 
Let $T,S \in \mathcal{B} (X)$. Assume that for every $x\in X$ the vector $Tx$ belongs to the linear span of $x$ and $Sx$. Then $T$ = $\lambda I + \mu S$ for some $ \lambda, \mu \in \mathbb{C}$.
 \end{lemma}
 \begin{proof} See lemma 2.4 in \cite{ref17}.
 \end{proof}
The following lemma was introduced in \cite{ref2} for generalized product of operators, but in this paper we use another techniques to characterize rank one operators by the dimension of analytical core of Jordan product of operators.
\begin{lemma}\label{lemme3}
For a nonzero operator $ A\in \mathcal{B} (X)$, the following statements are equivalent.
 \begin{enumerate}
\item[(i)] $A$ is a rank one operator.
\item[(ii)] $\dim (K(TA+AT))\leq 2$ for all $T\in \mathcal{B} (X)$.
 \end{enumerate}
 \end{lemma}
\begin{proof}
For $ A \in \mathcal{B} (X)$, we consider the following operator $ R = A T+ TA $ with $ T \in \mathcal{B} (X) $ is an arbitrary operator. It is clear that we have $(i)\Longrightarrow (ii)$. For $(ii)\Longrightarrow (i)$ suppose that $ Rank (A) \geq 2$, and let us show that there exists $T \in \mathcal{B} (X)$ such that $ \dim (K(R)) \geq 3 $. We will distinguish two cases. \\\\
\textbf{Case 1.} If $ Rank (A) \geq 3$.\\ Let $y_{1},y_{2},y_{3} \in X$ be linearly independent vectors such that $y_{1}= Ax_{1}$, $y_{2}= Ax_{2}$ and $y_{3}= Ax_{3}$ where $x_1, x_{2},x_3$ $ \in X$ are obviously linearly independent. We will distinguish four cases.
\begin{itemize}
\item[(1)] If  $span \lbrace x_{1},x_{2},x_{3}\rbrace =  span \lbrace y_{1},y_{2},y_{3}\rbrace. $\\ Let $ T \in \mathcal{B} (X)$ be an operator satisfying  $ Ty_{1}= x_{1}, Ty_{2}= x_{2}$ and $Ty_{3}= x_{3}.$ As $AT=TA=I$ on $span \lbrace x_1,x_2,x_3 \rbrace$, then $Rx_{1}=2x_{1}$, $Rx_{2}=2x_{2}$ and $Rx_{3}=2x_{3}$. Hence $span \lbrace x_{1},x_{2},x_{3} \rbrace \subset N(R-2I) \subset K (R)$, and so $ \dim (K(R)) \geq 3 $.\\
\item[(2)] If  $ \dim (span \lbrace x_{1},x_{2},x_{3},y_{1},y_{2},y_{3}\rbrace) = 4$.\\ Let for example $span\lbrace x_{1},x_{2},x_{3},y_{1},y_{2},y_{3}\rbrace=$ $span \lbrace x_{1},y_{1},y_{2},y_{3}\rbrace$. Consider an operator $ T \in \mathcal{B} (X)$ satisfying $Ty_{1}= x_{1}, Ty_{2}= x_{2}, Ty_{3}= x_{3}$ and $Tx_{1} = 0 .$ For $ x_{2}= \alpha x_{1} + \beta y_{1} + \gamma y_{2} + \delta y_{3} $ and $ x_3 = \alpha^{'} x_{1} + \beta^{'} y_{1} + \gamma^{'} y_{2} + \delta^{'} y_{3} $ where $ \alpha, \beta, \gamma, \delta, \alpha^{'}, \beta^{'}, \gamma^{'}, \delta^{'} \in \mathbb{C}$, we get that $Rx_{2}= 2x_{2}- \alpha x_{1}$ and $ Rx_{3} = 2x_{3}- \alpha' x_{1}$. Since $ Rx_{1}= x_{1}$, then $ RM=M$ with $M =span\lbrace x_{1}, x_{2},x_{3}\rbrace$. Hence $ M \subset K(R)$, therefore $\dim (K(R)) \geq 3$.\\
\item[(3)] If $ \dim (span \lbrace x_{1},x_{2},x_{3},y_{1},y_{2},y_{3}\rbrace )= 5$.\\ Let for example $span \lbrace x_{1},x_{2},x_{3},y_{1},y_{2},y_{3}\rbrace =span\lbrace x_{1},x_{2},y_{1},y_{2},y_{3}\rbrace $. Take an operator  $T \in \mathcal{B} (X) $ satisfying $Ty_{1}= x_{1}, Ty_{2}= x_{2}, Ty_{3}= x_{3}$ and $ Tx_{1}= Tx_{2}=0 .$ For $x_{3}= \alpha x_{1} + \beta x_{2} + y $ where $ y \in$ Span$\lbrace  y_{1},y_{2},y_{3}\rbrace $ and $ \alpha, \beta \in \mathbb{C}$, we get that $Rx_{1}=x_{1}, Rx_{2}=x_{2}$ and $ Rx_{3}= 2 x_{3} - \alpha x_{1} - \beta x_{2}  $. It follows that $ RM=M$ with $M =span\lbrace x_{1}, x_{2},x_{3}\rbrace$. Then $ M \subset K(R)$, thus $\dim (K(R)) \geq 3$. \\
\item[(4)] If $ \dim (span \lbrace x_{1},x_{2},x_{3},y_{1},y_{2},y_{3}\rbrace) = 6.$\\ Let $ T \in \mathcal{B} (X)$ be an operator satisfying $Ty_{1}= x_{1}, Ty_{2}= x_{2}, Ty_{3}= x_{3}$ and $ Tx_{1}= Tx_{2}= Tx_{3}= 0 .$ Since $Rx_{1}=x_{1}$, $Rx_{2}=x_{2}$ and $Rx_{3}=x_{3}$. It follows that $ span\lbrace x_{1},x_{2},x_{3}\rbrace \subset N(R-I) \subset K (R)$, and so $ \dim ( K(R)) \geq 3 $. 
\end{itemize}
\vspace{0.2cm}
\textbf{Case 2.} If $ Rank(A) = 2$.\\ Let $y_{1}$ and $y_2$ two linearly independent vectors such that $y_{1}=Ax_{1}$, $y_{2}=Ax_{2}$ and $Ay_1 = a y_1 +b y_2$ where $a,b \in \mathbb{C}$.\\ If $x_1,y_1 ,y_2   \in X$ are linearly dependent, we can choose $ u \in N(A)$ such that $x_1 +u,y_1 ,y_2\in X$ are linearly dependent. Now, if $x_1+u,x_2,y_1\in X$ are linearly dependent. As $x_1+u$ and $y_1$ are linearly independent, and $N(A)$ is an infinite-dimensional subspace of $X$. Then we can take $ v \in N(A)$ such that $x_1 +u,x_2 +v,y_1 \in X$ are linearly independent. Thus without loss of generality, we may assume that $( x_1, y_1,y_2 )$ and $( x_1, x_2,y_1 )$ are linearly independent. Let us distinguish two cases. 
\begin{itemize}
\item[(1)] If $  x_{2} \in$  $span\lbrace x_{1}, y_{1},y_{2} \rbrace .$\\ Let $  x_{2}= \alpha x_{1} + \beta y_{1} + \gamma y_{2}$ where $\alpha,\beta,\gamma \in \mathbb{C}$. For an operator $T\in \mathcal{B} (X)$ satisfying $Ty_{1}= x_{1}, Ty_{2}= x_{2}$ and $Tx_{1}= 0$, we get that
$$ \begin{cases}
Rx_{1}=x_{1} \\
Rx_{2}= 2x_{2}- \alpha x_{1} \\ 
Ry_{1}=y_{1}+ ax_{1}+ b x_{2}.\\
\end{cases} $$
This yields that $R M = M$ with $ M =$  $span\lbrace x_{1}, x_{2}, y_{1} \rbrace$. Hence $M \subset K(R)$, thus $ \dim (K(R)) \geq 3 $. \\
\item[(2)] If $ ( x_{1}, x_{2}, y_{1}, y_{2} ) $ are linearly independent.\\ Consider an operator $T\in \mathcal{B} (X)$ satisfying $Ty_{1}= x_{1}, Ty_{2}= x_{2}$ and $Tx_{1}=Tx_{2}= 0.$ Hence $$ \begin{cases} Rx_{1}=x_{1} \\
Rx_{2}=x_{2} \\
 Ry_{1}=y_{1}+ a x_{1}+ b x_{2}. \end{cases} $$ It follows that $R M = M$ with $ M =$ $span\lbrace x_{1}, x_{2}, y_{1} \rbrace $. Thus $M \subset K(R)$, and so $ \dim (K(R)) \geq 3 $. This establishes the lemma. 
 \end{itemize}
 \end{proof}
For Jordan product of operators, we give a characterization for two operators to be linearly dependent.
 \begin{lemma}\label{gene}
For $A$ and $B$ two operators in $\mathcal{B} (X) $, the following statements are equivalent. 
\begin{enumerate}
\item[(i)] $ K( A T + T A) = K (BT + TB) $ for every $ T \in \mathcal{B} (X) $.
\item[(ii)] $ K( A F + F A) = K (BF + FB) $ for every $ F \in \mathcal{F}_1 (X) $. 
\item[(iii)]  $ B = \lambda A $ for some nonzero scalar $ \lambda \in \mathbb{C}$.
 \end{enumerate}
\end{lemma}
\begin{proof} For $ A,B \in \mathcal{B} (X)$, we consider the following operators $ R = A F+ FA $ and $ S = B F + F B $ with $ F \in \mathcal{F}_1 (X)$ is an arbitrary operator. The implications (i) $ \Longrightarrow$ (ii) and (iii) $ \Longrightarrow$ (i) are straightforward. We only need to show the implication (ii) $ \Longrightarrow$ (iii). Let us distinguish two cases.\\\\ 
$\hspace*{0.3cm}$\textbf{Case 1.} $(x,Ax,Bx)$ are linearly independent for some nonzero vector $ x \in X $. Let $f$ be a linear functional in $X^*$ such that $f(x)= f (Ax) = 0 $ and $f(Bx)=1$. For $ F = x \otimes f $, we obtain that
$$ \begin{cases} Rx=0\\
RAx = f(A^{2}x)x 
\end{cases}
\quad\mbox{and \;\;\;\;\; }
 \begin{cases}
Sx = x\\
 SBx = f(B^{2}x)x + Bx. \\ 
\end{cases}$$
This yields that $R^{2}= 0 $, thus $K (R) = \lbrace 0 \rbrace $ but $ x \in K (S) $, which is a contradiction.\\
$\hspace*{0.3cm}$\textbf{Case 2.} $(x,Ax,Bx)$ are linearly dependent for all $x \in X$.\\ Let $x$ be a vector in $X$. If $ x$ and $Ax$ are linearly independent, then we get that $ Bx \in span \{ x,Ax \}$. If not, then $Ax= \mu x$ where $ \mu \in \mathbb{C}$. Suppose that $Bx$ and $x$ are linearly independent. It follows that there exists a linear functional $ f \in X^* $ such that $ f(x) = 0$ and $ f(Bx)=1 $. For $T = x\otimes f$ we obtain that $ Rx = 0$ and $ Sx=x$. Thus $x \in K(S) = K(R)= \lbrace 0 \rbrace $. This contradiction shows that $Bx \in span \lbrace x \rbrace \subset  span \{ x,Ax \}$. We conclude that $ Bx \in span \{ x,Ax \}$ for all $x\in X$. Lemma \ref{lemme2} implies that $ B = \lambda A + \alpha I $ where $\alpha , \lambda \in \mathbb{C} $.
It is clear that we have $$ A =0 \mbox{ if and only if } B=0$$ and also $$  A \in \mathbb{C}I  \mbox{ if and only if } B \in \mathbb{C}I.$$  Now, if $A \notin \mathbb{C}I $, assume that $ \alpha \neq 0 $. For that we distinguish two cases. 
\begin{itemize}
\item[(1)]  If $(x, Ax,A^{2}x)$ are linearly independent for some $x\in X$.\\ Take $ F = x \otimes f$ such that $f(x) = 1$ and $ f(Ax)=f(A^{2}x)=0$. Hence $R ^{2} = 0$, thus $ K (R) = \lbrace 0 \rbrace$. On the other hand, we have $K(S)= span \lbrace \lambda A x + 2 \alpha x \rbrace $, this is a contradiction.
\item[(2)] If $(x, Ax,A^{2}x)$ are linearly dependent for all $x\in X$.\\ It is obvious that we have $ A^2x \in span \lbrace x, Ax\rbrace $ for all $x\in X$. It follows by Lemma \ref{lemme2} that $ A^{2}= a A + b I$ where $a, b \in \mathbb{C}$. Without loss of generality, we may assume that $\lambda = 1 $. Since $ A \not\in \mathbb{C}I$. Then there exists a nonzero vector $x \in X$ such that $Ax$ and $x$ are linearly independent. Consider an operator $F=x \otimes f$ such that $ f (x) \neq 0 $ and $ f(Ax)^{2}=f(x)f(A^{2}x)$, then we obtain $$ \begin{cases} Rx=f(Ax)x + f(x)Ax\\
RAx =f (A^{2}x) x + f(Ax) Ax \\    
\end{cases}
\quad\mbox{and \;\;\;\;\; }
 \begin{cases}
Sx = f(Bx)x + f(x)Bx  \\
 SAx = f(BAx)x + f(Ax)Bx. \\ 
\end{cases}$$
Hence $$
 \begin{cases}
 Rx=f(Ax)x + f(x)Ax\\
RAx =\dfrac{f(Ax)}{f(x)} Rx
\end{cases}
\quad\mbox{and \;\;\;\;\; }
 \begin{cases}
Sx = f(Ax)x + f(x) Ax + 2 \alpha f(x)x  \\
 SAx = (f(A^{2}x) + \alpha f(Ax))x + f(Ax)(Ax + \alpha x). \\ 
\end{cases}$$
Thus$$
 \begin{cases}
Sx = Rx + 2 \alpha f(x)x  \\
 SAx = f(A^{2}x) x + f(Ax)Ax + 2 \alpha f(Ax)x, \\ 
\end{cases}$$
and so
  $$
 \begin{cases}
Sx = Rx +  2\alpha f(x)x  \\
 SAx = \dfrac{f(Ax)}{f(x)} Sx. \\ 
\end{cases}$$

\begin{enumerate}
\item[$\bullet$] If $f(Ax)=0$, we have $R^{2} = 0$ and $ K(R)=\lbrace 0 \rbrace $. Since $ Sx = f(x)Ax + 2\alpha f(x) x$ and $SAx=0$, then $S^{2}x= 2 \alpha  f(x) Sx$. This yields that $Sx \in K(S) = K(R)=\lbrace 0 \rbrace$, contradiction.
\item[$\bullet$] Now, if $f(Ax) \neq 0$, we have $K(R) = $ $span \lbrace Rx \rbrace$ and $ K(S) = $ $span \lbrace Sx  \rbrace$. Hence there exists a nonzero scalar $\delta \in \mathbb{C} $ such that $\; Sx = \delta Rx $. This implies that $ 2 \alpha f(x) x = (\delta-1)Rx $. As $x$ and $Ax$ are linearly independent, then $\delta = 1$, contradiction.
\end{enumerate}
\end{itemize}
Finally, we get that $ \alpha = 0 $ and the proof is complete. 
\end{proof}

 \section{Main Result} 
We start this section with the following theorem that gives a characterization of maps preserving the analytic core of the Jordan product of operators.
 \begin{theorem}\label{thm1}
Let $\phi: \mathcal{B} (X) \longrightarrow \mathcal{B} (X)$ be a surjective map. Then the following assertions are equivalent:
\begin{enumerate}
\item[(i)] For every $ T, S \in \mathcal{B} (X)$, we have 
$$K(\phi (T) \phi (S) + \phi (S) \phi (T)) = K(TS + ST) $$

\item[(ii)] There is a map $ \gamma : \mathcal{B} (X) \longrightarrow \mathbb{C} \backslash \lbrace 0 \rbrace$ 
such that $ \phi (T) = \gamma (T) T $ for all $ T \in \B (X)$.
 \end{enumerate}
\end{theorem}
\begin{proof} Clearly, we have $(ii) \Longrightarrow (i) $. It remains to show that $ (i) \Longrightarrow  (ii) $. We divide the proof into several steps.\\\\
\textbf{Step 1.} For any $R \in \mathcal{B} (X)$, $ \phi (R) = 0 $ if and only if $ R = 0. $ \\ \\
Let us show that $ \phi (0) = 0$. We have \begin{eqnarray*}
K (\phi(T) \phi (0) + \phi (0)\phi(T)) &=& K (T0+0T) \\
&=& \lbrace 0 \rbrace\\
&=& K(\phi(T) 0 + 0 \phi (T)).
\end{eqnarray*}
As $ \phi $ is surjective, then Lemma \ref{gene} implies that $\phi (0) = 0$.\\
Now, assume that $ \phi (R)= 0$ for an operator $ R \in \mathcal{B} (X) $. Hence \begin{eqnarray*}
K (R T + TR) &=& K (\phi(R) \phi(T) + \phi (T) \phi (R)) \\
&=& \lbrace 0 \rbrace\\
&=& K( 0 T + T 0 ).
\end{eqnarray*}
 From Lemma \ref{gene}, we infer that $ R = 0$.\\ \\
\textbf{Step 2.} For any operator $ F \in \mathcal{B} (X)$, $\phi (F) \in \mathcal{F}_1 (X)$ if and only if $F \in \mathcal{F}_1 (X)$.\\ \\
Let $ F \in \mathcal{B} (X)$ such that $ \phi (F)$ is a rank one operator. Step 1 implies that $ F \neq 0$. By Lemma \ref{lemme3}, we obtain that $$ \dim ( K(S \phi (F) + \phi(F) S))  \leq 2$$ for all $ S \in \mathcal{B} (X)$. As $\phi$ is surjective, then $$ \dim ( K(TF+FT)) = \dim ( K(\phi (T) \phi (F) + \phi(F) \phi (T))) \leq 2 $$ for all $ T \in \mathcal{B} (X)$. It follows from Lemma \ref{lemme3} that $ F$ is a rank one operator. Just as before, we obtain in the same way that if $F \in \mathcal{F}_1 (X)$ then  $ \phi (F) \in \mathcal{F}_1 (X)$. \\\\ 
\textbf{Step 3.} There is a nonzero scalar $ \lambda_{F} \in \mathbb{C} $ such that $ \phi (F) = \lambda_{F} F $ for all $ F \in \mathcal{F}_1 (X)$. For that we discuss two cases, if $ F $ is nilpotent or not.
\begin{itemize}
\item[(1)] There is a nonzero scalar $\lambda_{P}$ such that $\phi (P) = \lambda_{P} P $, for every rank one idempotent operator $P \in \mathcal{F}_1 (X)$.\\
Let $ x \in X $ and $ f \in X^{*}$ such that $ f(x) = 1 $. Step 2 implies that there exist a nonzero vector $ y \in X $ and a linear functional $ g \in X^{*}$ such that $ \phi (x\otimes f )= y \otimes g $. Since $x \otimes f$ is a rank one idempotent operator, then
\begin{eqnarray*}
 span \lbrace x \rbrace &=& K (x \otimes f ) \\
 &=& K (x \otimes f x \otimes f + x \otimes f x \otimes f )\\
 &=& K (y \otimes g y \otimes g +y \otimes g x \otimes g )\\ 
 &=& K ( 2 g(y) y \otimes g ).
\end{eqnarray*} 
This yields that $ g(y) \neq 0 $ and $  span \lbrace y \rbrace =  K ( 2 g(y) y \otimes g ) = span \lbrace x \rbrace $. Therefore $ y = \alpha x $ for a nonzero scalar $\alpha \in \mathbb{C}$. Without loss of generality, we may and shall assume that $ \alpha = 1$, then we get that $ \phi (x \otimes f) = x \otimes g_{x,f}$ for certain linear functional $   g_{x,f} \in X^{*}$.
We claim that $ f$ and $g_{x,f}$ are linearly dependent. If not, take a nonzero vector $ z \in X$ such that $x$ and $z$ are linearly independent in X, with $f(z) =1$ and $g_{x,f}(z)=0$. Just as before, one shows that there exists a linear functional $ g_{z,f} \in X^{*}$ such that $ \phi (z \otimes f) = z \otimes g_{z,f}$. Observe that $$ (x \otimes f z \otimes f + z \otimes f x \otimes f  )( x + z ) = 2 ( x + z ).$$ Then $$ (x+z) \in  K(x \otimes f z \otimes f + z \otimes f x \otimes f ).$$ On the other hand, we have
\begin{eqnarray*}
\lbrace 0 \rbrace&=&K(g_{z,f}(x)z \otimes g_{x,f}) \\
&=& K(x \otimes g_{x,f} z \otimes g_{z,f} + z \otimes g_{z,f} x \otimes g_{x,f} )  \\
&=& K(\phi(x \otimes f)\phi( z \otimes f) + \phi(z \otimes f )\phi(x \otimes f ))\\
&=& K(x \otimes f z \otimes f + z \otimes f x \otimes f ), 
\end{eqnarray*}
which is a contradiction. Hence $ g_{x,f}$ and $f$ are linearly dependent, thus $\phi (x \otimes f)= \alpha x \otimes f  $ for some nonzero scalar $\alpha \in \mathbb{C}$.
\item[(2)] There is a nonzero scalar $\lambda_{N}$ such that $\phi (N) = \lambda_{N} N $, for every rank one nilpotent operator $N \in \mathcal{F}_1 (X)$.\\
Let $ x \in X $ and $ f \in X^{*}$ such that $ f(x) = 0 $. Set $ T = x \otimes f$, then $ \phi (T) = y \otimes g $ with $ g(y) = 0$. Assume that $x$ and $y$ are linearly independent, and so let $ z \in X$ and $ h \in X^{*}$ such that $ f(z) = h(x) = 1$, $ h(z) \neq 0$ and $ h(y) = 0$. It follows that 
\begin{eqnarray*}
\lbrace 0 \rbrace &=& K(g(z) y \otimes h) \\
&=& K(y \otimes g z \otimes h + z \otimes h y \otimes g)\\
&=&K(\phi(x \otimes f )\phi(z \otimes h) + \phi(z \otimes h )\phi(x \otimes f)) \\
&=& K(x \otimes f z \otimes h + z \otimes h x \otimes f).
\end{eqnarray*}
Observe that $$(x \otimes f z \otimes h + z \otimes h x \otimes f) x = x ,$$ hence $$span\lbrace x \rbrace \subset K(x \otimes f z \otimes h + z \otimes h x \otimes f).$$ This contradiction shows that $x$ and $y$ are linearly dependent. As before, without loss of generality, we may assume that $ \phi (x \otimes f) = x \otimes g_{x,f}$ for certain linear functional $   g_{x,f} \in X^{*}$.\\
Suppose that $f$ and $g$ are linearly independent, and let $z \in X$ and $ h \in X^{*}$ such that $f(z) = 1 $, $ g(z)=0$, $ h(x) = 1$ and $ h(z) \neq 0$. Since $$(x \otimes f z \otimes h + z \otimes h x \otimes f) x = x,$$ then $$span\lbrace x \rbrace \subset K(x \otimes f z \otimes h + z \otimes h x \otimes f).$$ On the other hand, we get that
\begin{eqnarray*}
\lbrace 0 \rbrace &=& K(z \otimes g) \\
&=& K(x \otimes g z \otimes h + z \otimes h x \otimes g)\\
&=&K(\phi(x \otimes f )\phi(z \otimes h) + \phi(z \otimes h )\phi(x \otimes f)) \\
&=& K(x \otimes f z \otimes h + z \otimes h x \otimes f).
\end{eqnarray*}
Then, this contradiction shows also that $f$ and $g$ are linearly dependent. Thus $ \phi (x \otimes f) = \beta x \otimes f$ for some nonzero scalar $\beta \in \mathbb{C}$. \\
Now, if $ f(x) \in \mathbb{C}\setminus \{0,1\}$, obviously this can be obtained as in (1) of Step 3. Finally, there is a nonzero scalar $\lambda_{F} \in \mathbb{C}$ such that $\phi (F)= \lambda_{F} F$ for all rank one operator $ F \in \mathcal{F}_1 (X)$. \\
\end{itemize}
 \textbf{Step 4.} There exists a map $ \gamma: \mathcal{B} (X) \longrightarrow \mathbb{C} \backslash \lbrace 0 \rbrace$ such that $ \phi (T) = \gamma (T) T$ for all $ T \in \mathcal{B} (X) $.\\ \\
For every $F \in \mathcal{F}_1 (X)$ and every $T \in \mathcal{B} (X)$, we have $$K (TF + FT) = K (\phi (T) \phi (F) + \phi (F) \phi (T)) .$$ By using step 3, we get that
\begin{eqnarray*}
K (\phi (T) \phi (F) + \phi (F) \phi (T)) &=& K (\lambda_F (\phi (T) F +  F \phi (T))) \\
 &=& K ( \phi (T) F +  F \phi (T)).
\end{eqnarray*}
Then $K (TF + FT) = K ( \phi (T) F +  F \phi (T))$, for all operator $ F \in \mathcal{F}_1 (X) $.
Lemma \ref{gene} implies that $ \phi (T)$ and $T$ are linearly dependent. Hence, there exists a map $ \gamma: \mathcal{B} (X) \longrightarrow \mathbb{C} \backslash \lbrace 0 \rbrace$ such that $ \phi (T) = \gamma (T) T$ for all $ T \in \mathcal{B} (X) $. The proof is then complete. 
  \end{proof} 
As a consequence of theorem \ref{thm1}, we get the following corollary.
  \begin{corollary}
  Let $\phi: \mathcal{B} (X) \longrightarrow \mathcal{B} (X)$ be a surjective map. Then the following assertions are equivalent.
\begin{enumerate}
\item[(i)] $i_{\phi (T) \phi (S) + \phi (S) \phi (T)}(x)=0 \Longleftrightarrow i_{TS + ST} (x) = 0 $ for every $ T, S \in \mathcal{B} (X)$ and $ x \in X$. 
\item[(ii)] There exists a map $ \gamma : \mathcal{B} (X) \longrightarrow \mathbb{C} \backslash \lbrace 0 \rbrace$ 
such that $ \phi (T) = \gamma (T) T $ for all $ T \in \mathcal{B} (X)$.
\end{enumerate}
 \end{corollary}
\begin{proof}
The implication $(ii) \Longrightarrow (i) $ is straightforward.\\
$(i) \Longrightarrow (ii)$ 	From theorem \ref{thm1}, and the assertions $(vi)$ and $(vii)$ in lemma \ref{lem1}, we get the desired form of $\phi$.
\end{proof} 



\end{document}